\documentclass[a4paper,10pt,reqno]{amsart}
\usepackage{amsmath}
\usepackage{amssymb}
\usepackage{amsthm}
\usepackage[all]{xy}
\usepackage{hyperref}

\newlength{\doubleofsidemargin}
\newlength{\sidemargin}
\theoremstyle{definition}
\newtheorem{Def}{Definition}[section]
\newtheorem{rem}[Def]{Remark}

\newtheorem{ex}[Def]{Example}

\theoremstyle{plain}
\newtheorem{Prop}[Def]{Proposition}
\newtheorem{Lem}[Def]{Lemma}
\newtheorem{Thm}[Def]{Theorem}
\newtheorem{Cor}[Def]{Corollary}
\newtheorem{MThm}[Def]{Main Theorem}

\newcommand{\SwapSymbols}[1]{
	\expandafter\let\expandafter\temporarysymbol\csname #1\endcsname
	\expandafter\let\csname #1\expandafter\endcsname\csname var#1\endcsname
	\expandafter\let\csname var#1\endcsname\temporarysymbol
}

\SwapSymbols{Gamma}
\SwapSymbols{Delta}
\SwapSymbols{Theta}
\SwapSymbols{Lambda}
\SwapSymbols{Xi}
\SwapSymbols{Pi}
\SwapSymbols{Sigma}
\SwapSymbols{Upsilon}
\SwapSymbols{Phi}
\SwapSymbols{Psi}
\SwapSymbols{Omega}

\def\mbZ{\mathbb{Z}}

\def\mcA{\mathcal{A}}

\def\mcC{\mathcal{C}}
\def\mcF{\mathcal{F}}

\def\mcL{\mathcal{L}}

\def\mcS{\mathcal{S}}
\def\mcX{\mathcal{X}}

\def\mfp{\mathfrak{p}}

\def\H{{\rm H}}
\def\R{{\mathbf{R}}}
\def\D{{\mathbf{D}}}

\makeatletter
\def\Mod{\mathop{\operator@font Mod}\nolimits}
\def\mod{\mathop{\operator@font mod}\nolimits}
\def\noeth{\mathop{\operator@font noeth}\nolimits}
\def\Hom{\mathop{\operator@font Hom}\nolimits}
\def\End{\mathop{\operator@font End}\nolimits}
\def\Ext{\mathop{\operator@font Ext}\nolimits}
\def\Ker{\mathop{\operator@font Ker}\nolimits}
\def\Im{\mathop{\operator@font Im}\nolimits}
\def\Cok{\mathop{\operator@font Cok}\nolimits}
\def\RHom{\mathop{\operator@font \mathbf{R}Hom}\nolimits}

\def\Spec{\mathop{\operator@font Spec}\nolimits}
\def\Max{\mathop{\operator@font Max}\nolimits}
\def\Supp{\mathop{\operator@font Supp}\nolimits}
\def\Ass{\mathop{\operator@font Ass}\nolimits}
\def\ASpec{\mathop{\operator@font ASpec}\nolimits}
\def\ASupp{\mathop{\operator@font ASupp}\nolimits}
\def\AAss{\mathop{\operator@font AAss}\nolimits}
\def\asupp{\mathop{\operator@font asupp}\nolimits}

\def\Zg{\mathop{\operator@font Zg}\nolimits}
\def\Ann{\mathop{\operator@font Ann}\nolimits}
\makeatother

\title[Extension groups between atoms and objects]{Extension groups between atoms and objects in locally noetherian Grothendieck category}
\author{Ryo Kanda}
\thanks{The author is a Research Fellow of Japan Society for the Promotion of Science. This work is supported by Grant-in-Aid for JSPS Fellows 25$\cdot$249.}
\address{Graduate School of Mathematics, Nagoya University, Furo-cho, Chikusa-ku, Nagoya-shi, Aichi-ken, 464-8602, Japan}
\email{kanda.ryo@a.mbox.nagoya-u.ac.jp}

\begin{document}

\begin{abstract}
	We define the extension group between an atom and an object in a locally noetherian Grothendieck category as a module over a skew field. We show that the dimension of the $i$-th extension group between an atom and an object coincides with the $i$-th Bass number of the object with respect to the atom. As an application, we give a bijection between the E-stable subcategories closed under arbitrary direct sums and direct summands and the subsets of the atom spectrum and show that such subcategories are also closed under extensions, kernels of epimorphisms, and cokernels of monomorphisms. We show some relationships to the theory of prime ideals in the case of noetherian algebras.
\end{abstract}

\maketitle

\section{Introduction}
\label{sec:intro}

In this paper, we investigate Bass numbers in a locally noetherian Grothendieck category and give generalizations of fundamental results in the commutative ring theory.

Let $R$ be a commutative noetherian ring. For a nonnegative integer $i$, the $i$\emph{-th Bass number} $\mu_{i}(\mfp,M)$ of an $R$-module $M$ with respect to a prime ideal $\mfp$ of $R$ is defined by
	\begin{equation*}
		E^{i}(M)=\bigoplus_{\mfp\in\Spec R}E(R/\mfp)^{\oplus\mu_{i}(\mfp,M)},
	\end{equation*}
	where $E^{i}(M)$ denotes the $i$-th term in the minimal injective resolution of $M$, and $E(R/\mfp)$ is the injective envelope of $R/\mfp$. Bass numbers are important invariants of modules in order to study homological aspects of commutative rings. For example, Bass \cite{Bass63} proved the following theorem.

\begin{Thm}[{Bass \cite[Lemma 2.7]{Bass63}}]\label{Thm:Bassthm}
	Let $R$ be a commutative noetherian ring, $\mfp$ a prime ideal of $R$, $M$ an $R$-module, and $i$ a nonnegative integer. Then we have the equation
	\begin{equation*}
		\mu_{i}(\mfp,M)=\dim_{k(\mfp)}\Ext_{R_{\mfp}}^{i}(k(\mfp),M_{\mfp})=\dim_{k(\mfp)}\Ext_{R}^{i}(R/\mfp,M)_{\mfp},
	\end{equation*}
	where $k(\mfp)$ is the residue field of $\mfp$.
\end{Thm}

In the case of noncommutative rings, prime ideals do not always work well, and hence it is hard to generalize Theorem \ref{Thm:Bassthm} to noncommutative rings straightforwardly. In this paper, we see that the generalization is possible if we treat it in a viewpoint of atoms.

For a ring $R$, Storrer \cite{Storrer} introduced the notion of atoms in the category $\Mod R$ of right $R$-modules, which are equivalence classes of monoform modules. In the case where $R$ is a right noetherian ring, it is known that the atoms bijectively correspond to the isomorphism classes of indecomposable injective modules. If $R$ is a commutative ring (which is not necessarily noetherian), the atoms in $\Mod R$ bijectively correspond to the prime ideals of $R$. In \cite{Kanda}, we investigated atoms in arbitrary abelian categories, especially noetherian abelian categories and locally noetherian Grothendieck categories, and gave classifications of Serre subcategories and localizing subcategories, respectively.

In the case of locally noetherian Grothendieck category $\mcA$, the atoms also bijectively correspond to the isomorphism classes of simple objects in the spectral category of $\mcA$, which was introduced by Gabriel and Oberst \cite{GabrielOberst}. In this paper, we regard a notion in \cite{GabrielOberst} as morphism spaces between atoms and objects (Definition \ref{Def:morspace}), and by deriving them, define extension groups between atoms and objects (Definition \ref{Def:extgroup}). Then we obtain the following description of Bass numbers (Definition \ref{Def:Bassnumber}) for $\mcA$.

\begin{MThm}[Theorem \ref{Thm:Bassext}]\label{MThm:main}
	Let $\mcA$ be a locally noetherian Grothendieck category, $\alpha$ an atom in $\mcA$, $M$ an object in $\mcA$, and $i$ a nonnegative integer. Then we have the equation
	\begin{equation*}
		\mu_{i}(\alpha,M)=\dim_{k(\alpha)}\Ext^{i}_{\mcA}(\alpha,M),
	\end{equation*}
	where $\mu_{i}(\alpha,M)$ is the $i$-th Bass number of $M$ with respect to $\alpha$, and $k(\alpha)$ is the residue (skew) field of $\alpha$.
\end{MThm}

This theorem has some applications to E-stable subcategories. A full subcategory $\mcX$ of $\mcA$ is called $E$\emph{-stable} if for any object $M$ in $\mcA$, $M$ belongs to $\mcX$ if and only if $E^{i}(M)$ belongs to $\mcX$ for each nonnegative integer $i$. As a generalization of a result by Takahashi \cite[Theorem 2.18]{Takahashi}, we show the following classification.

\begin{Thm}[Theorem \ref{Thm:subcatatom}]\label{Thm:introsubcatatom}
	Let $\mcA$ be a locally noetherian Grothendieck category. Then there exists a bijection between the E-stable subcategories of $\mcA$ closed under arbitrary direct sums and direct summands, and the subsets of the set of all the atoms in $\mcA$.
\end{Thm}

As an application of Main Theorem \ref{MThm:main}, we show the following result. In the case of commutative noetherian rings, it is stated by Takahashi \cite[Corollary 2.19]{Takahashi}.

\begin{Cor}[Corollary \ref{Cor:subcatexact}]
	Let $\mcA$ be a locally noetherian Grothendieck category and $\mcX$ an E-stable subcategory of $\mcA$ closed under arbitrary direct sums and direct summands. Then $\mcX$ is also closed under extensions, kernels of epimorphisms, and cokernels of monomorphisms.
\end{Cor}

In section \ref{sec:modfinalg}, we deal with noetherian algebras and show some relationships between our theory and the theory of prime ideals described in \cite{GotoNishida}.

We refer the reader to \cite{Matsumura} for the general theory of commutative rings, \cite{Popescu} for that of abelian categories and Grothendieck categories, and \cite{Weibel} for that of derived categories.

\section{Atoms}
\label{sec:atoms}

Throughout this paper, we deal with a locally noetherian Grothendieck category $\mcA$. Its definition is as follows.

\begin{Def}\leavevmode
	\begin{enumerate}
		\item An abelian category $\mcA$ is called a \emph{Grothendieck category} if $\mcA$ has a generator and arbitrary direct sums and satisfies the following condition: for any object $M$ in $\mcA$, any family $\mcL=\{L_{\lambda}\}_{\lambda\in\Lambda}$ of subobjects of $M$ such that any finite subfamily of $\mcL$ has an upper bound in $\mcL$, and any subobject $N$ of $M$, we have
		\begin{equation*}
			\left(\sum_{\lambda\in\Lambda}L_{\lambda}\right)\cap N=\sum_{\lambda\in\Lambda}(L_{\lambda}\cap N).
		\end{equation*}
		\item A Grothendieck category $\mcA$ is called \emph{locally noetherian} if there exists a set of generators of $\mcA$ consisting of noetherian objects.
	\end{enumerate}
\end{Def}

\begin{rem}\leavevmode
	\begin{enumerate}
		\item It is known that an abelian category with a generator and arbitrary direct sums is a Grothendieck category if and only if direct limit is exact.
		\item For any Grothendieck category $\mcA$, it is shown in \cite[Theorem 2.9]{Mitchell} that any object $M$ in $\mcA$ has its injective hull $E(M)$ in $\mcA$.
	\end{enumerate}
\end{rem}

We recall the definition of atoms, which was introduced by Storrer \cite{Storrer}.

\begin{Def}
	Let $\mcA$ be a locally noetherian Grothendieck category.
	\begin{enumerate}
		\item An object $H$ in $\mcA$ is called \emph{monoform} if for any nonzero subobject $N$ of $H$, there exists no common nonzero subobject of $H$ and $H/N$, that is, there does not exist a nonzero subobject of $H$ which is isomorphic to a subobject of $H/N$.
		\item We say that monoform objects $H$ and $H'$ in $\mcA$ are \emph{atom-equivalent} if there exists a common nonzero subobject of $H$ and $H'$.
	\end{enumerate}
\end{Def}

The following properties about monoform objects are well known (for example, \cite{Storrer}). For complete proofs, we refer to \cite{Kanda}.

\begin{Prop}\label{Prop:monoform}
	Let $\mcA$ be a locally noetherian Grothendieck category.
	\begin{enumerate}
		\item\label{item:monoformsub} Any nonzero subobject of a monoform object in $\mcA$ is also monoform.
		\item\label{item:monoformuniform} Any monoform object $H$ in $\mcA$ is uniform, that is, any nonzero subobjects $N_{1}$ and $N_{2}$ of $H$ have a nonzero intersection.
		\item\label{item:submonoform} Any nonzero object in $\mcA$ has a monoform subobject.
	\end{enumerate}
\end{Prop}

\begin{proof}
	(1) \cite[Proposition 2.2]{Kanda}.
	
	(2) \cite[Proposition 2.6]{Kanda}.
	
	(3) \cite[Theorem 2.9]{Kanda}. Note that any object in $\mcA$ has a nonzero noetherian subobject.
\end{proof}

By Proposition \ref{Prop:monoform} (\ref{item:monoformuniform}), the atom equivalence is an equivalence relation between monoform objects in $\mcA$. We denote by $\ASpec\mcA$ the quotient set and call it the \emph{atom spectrum} of $\mcA$, and we call an element of $\ASpec\mcA$ an \emph{atom} in $\mcA$. The equivalence class of a monoform object $H$ is denoted by $\overline{H}$, and we regard $H$ as an element of $\overline{H}$. As in \cite[Proposition 5.3]{Kanda}, $\ASpec\mcA$ forms a set. In the case where $R$ is a commutative noetherian ring, the map $\mfp\mapsto\overline{R/\mfp}$ is a bijection between the prime spectrum $\Spec R$ of $R$ and $\ASpec(\Mod R)$. For the details of these arguments, see \cite{Storrer} or \cite{Kanda}.

In the commutative ring theory, Matlis \cite[Proposition 3.1]{Matlis} shows that for a commutative noetherian ring $R$, the map $\mfp\mapsto E(R/\mfp)$ is a bijection between $\Spec R$ and the set of all the isomorphism classes of indecomposable injective modules over $R$. This result can be generalized in terms of atoms as follows. That is shown by Storrer \cite[Corollary 2.5]{Storrer} in the case of a module category, and a complete proof for the case of a locally noetherian Grothendieck category is in \cite[Theorem 5.9]{Kanda}.

\begin{Thm}[{Storrer \cite{Storrer}}]\label{Thm:atominj}
	Let $\mcA$ be a locally noetherian Grothendieck category.
	\begin{enumerate}
		\item\label{item:injenvofatom} Let $\alpha$ be an atom in $\mcA$ and $H\in\alpha$. Then the injective hull $E(H)$ of $H$ does not depend on the choice of the representative $H$ of $\alpha$, that is, for another $H'\in\alpha$, $E(H')$ is isomorphic to $E(H)$. Hence denote it by $E(\alpha)$ and call it the injective hull of $\alpha$.
		\item\label{item:atomindecinj} The map $\alpha\mapsto E(\alpha)$ is a bijection between $\ASpec\mcA$ and the set of all the isomorphism classes of indecomposable injective objects in $\mcA$.
	\end{enumerate}
\end{Thm}

In general, for an atom $\alpha$ in $\mcA$, there exist many monoform objects $H$ such that $H\in\alpha$. However, we can take a canonical one $H_{\alpha}$ in the following sense.

\begin{Thm}\label{Thm:atomicobj}
	Let $\mcA$ be a locally noetherian Grothendieck category and $\alpha$ an atom in $\mcA$. Then there exists a unique maximal monoform subobject $H_{\alpha}$ of $E(\alpha)$.
\end{Thm}

\begin{proof}
	See \cite{Storrer} for the case of a module category. We give a proof for the general case.
	
	Denote by $J_{\alpha}$ the set of all the endomorphisms of $E(\alpha)$ which are not automorphisms. Since any Grothendieck category has arbitrary direct products (\cite[Corollary 7.10]{Popescu}), we can define a subobject $H_{\alpha}$ of $E(\alpha)$ by
	\begin{equation*}
		H_{\alpha}=\bigcap_{f\in J_{\alpha}}\Ker f,
	\end{equation*}
	where $\bigcap_{f\in J_{\alpha}}\Ker f$ is defined as the kernel of the morphism
	\begin{equation*}
		E(\alpha)\to\prod_{f\in J_{\alpha}}\frac{E(\alpha)}{\Ker f}
	\end{equation*}
	induced by the canonical morphism $E(\alpha)\to E(\alpha)/\Ker f$ for each $f\in J_{\alpha}$.
	
	Let $H$ be a monoform subobject of $E(\alpha)$ and $f\in J_{\alpha}$. Denote the composite $H\hookrightarrow E(\alpha)$ and $f:E(\alpha)\to E(\alpha)$ by $g$. If $f$ is a monomorphism, then $f$ is a split monomorphism which is not an epimorphism. This contradicts to the indecomposability of $E(\alpha)$. Hence $f$ is not a monomorphism. Assume that $H\not\subset\Ker f$. Then $\Im g\neq 0$. Since $E(\alpha)$ is uniform, we have $\Ker g=H\cap\Ker f\neq 0$ and $H\cap\Im g\neq 0$. Then we have
	\begin{equation*}
		\frac{H}{\Ker g}\cong \Im g\supset H\cap\Im g\subset H
 	\end{equation*}
	and this contradicts the monoformness of $H$. Therefore $H\subset\Ker f$, and this shows that $H\subset H_{\alpha}$.
	
	Assume that $H_{\alpha}$ is not monoform. Then there exist a nonzero subobject $N$ of $H_{\alpha}$ and a nonzero subobject $B$ of $H_{\alpha}$ which is isomorphic to a subobject $B'$ of $H_{\alpha}/N$. By the injectivity of $E(\alpha)$, we obtain a morphism $g'$ in $\mcA$ such that the diagram
	\begin{equation*}
		\newdir^{ (}{!/-5pt/@^{(}}
		\xymatrix{
			H_{\alpha}/N\ar^{g'}[r] & E(\alpha) \\
			& H_{\alpha}\ar@{^{ (}->}[u] \\
			B'\ar[r]^{\sim}\ar@{^{ (}->}[uu] & B\ar@{^{ (}->}[u]
		}
	\end{equation*}
	commutes. We also obtain a morphism $f'$ in $\mcA$ such that the diagram
	\begin{equation*}
		\newdir^{ (}{!/-5pt/@^{(}}
		\xymatrix{
			E(\alpha)\ar^{f'}[ddr] & \\
			H_{\alpha}\ar@{^{ (}->}[u]\ar@{>>}[d] & \\
			H_{\alpha}/N\ar^{g'}[r] & E(\alpha)
		}
	\end{equation*}
	commutes. Since $N\neq 0$, $f'$ is not a monomorphism, and hence $f'\in J_{\alpha}$. By the definition of $H_{\alpha}$, we have $H_{\alpha}\subset \Ker f'$. Then $g'=0$ and hence $B=0$. This is a contradiction. Therefore $H_{\alpha}$ is a unique maximal monoform subobject of $E(\alpha)$.
\end{proof}

We call the monoform object $H_{\alpha}$ in Theorem \ref{Thm:atomicobj} the \emph{atomic object} corresponding to the atom $\alpha$.

\section{Morphism spaces between atoms and objects}
\label{sec:morspaces}

Recall that a Grothendieck category $\mcS$ is called a \emph{spectral category} if any short exact sequence in $\mcS$ splits. Gabriel and Oberst \cite{GabrielOberst} defined the associated spectral category to a Grothendieck category $\mcA$, whose definition is based on the notion of essential subobjects.

Let $M$ be a nonzero object in $\mcA$. A subobject $N$ of $M$ is called an \emph{essential subobject} of $M$ if for any nonzero subobject $L$ of $M$, we have $L\cap N\neq 0$. Note that $M$ is uniform if and only if any nonzero subobject of $M$ is an essential subobject of $M$. We denote by $\mcF_{M}$ the set of all the essential subobjects of $M$. Since the intersection of two essential subobjects of $M$ is also an essential subobject of $M$, $\mcF_{M}$ is a directed set with respect to the opposite relation of the inclusion of subobjects.

\begin{Def}[Gabriel and Oberst {\cite{GabrielOberst}}]\label{Def:spectralcat}
	Define the \emph{spectral category} $\mcS$ of a Grothendieck category $\mcA$ as follows.
	\begin{enumerate}
		\item The objects of $\mcS$ are the same as the objects of $\mcA$.
		\item For objects $M$ and $N$ in $\mcA$,
		\begin{equation*}
			\Hom_{\mcS}(M,N)=\varinjlim_{M'\in\mcF_{M}}\Hom_{\mcA}(M',N).
		\end{equation*}
		\item Let $L,M,N$ be objects in $\mcA$ and $[f]\in\Hom_{\mcS}(L,M)$ and $[g]\in\Hom_{\mcS}(M,N)$. We assume that $[f]$ and $[g]$ are represented by $f\in\Hom_{\mcA}(L',M)$ and $g\in\Hom_{\mcA}(M',N)$, where $L'\in\mcF_{L}$, $M'\in\mcF_{M}$, respectively. Then the composite $[g][f]\in\Hom_{\mcS}(L,N)$ is the equivalence class of the composite of $f':f^{-1}(M')\to M'$ and $g:M'\to N$, where $f'$ is the restriction of $f$.
	\end{enumerate}
\end{Def}

Note that the inverse image of an essential subobject is also essential.

With the notation in Definition \ref{Def:spectralcat}, we can define a canonical additive functor $P:\mcA\to\mcS$ by the correspondence $M\mapsto M$ for each object $M$ in $\mcA$ and the canonical map $P_{M,N}:\Hom_{\mcA}(M,N)\to\varinjlim_{M'\in\mcF_{M}}\Hom_{\mcA}(M',N)$ for objects $M$ and $N$ in $\mcA$.

The following theorem states fundamental properties on the spectral category of $\mcA$.

\begin{Thm}[Gabriel and Oberst {\cite{GabrielOberst}}]\label{Thm:spectralcat1}
	Let $\mcA$ be a Grothendieck category, $\mcS$ the spectral category of $\mcA$, and $P:\mcA\to\mcS$ the canonical additive functor.
	\begin{enumerate}
		\item\label{item:spectralsplit} $\mcS$ is a spectral category, that is, $\mcS$ is a Grothendieck category such that any short exact sequence in $\mcS$ splits.
		\item\label{item:leftexactdirectsums} $P$ is a left exact functor and commutes with arbitrary direct sums.
	\end{enumerate}
\end{Thm}

Gabriel and Oberst \cite{GabrielOberst} also showed the following results.

\begin{Thm}[Gabriel and Oberst {\cite{GabrielOberst}}]\label{Thm:spectralcat2}
	Let $\mcA$ be a Grothendieck category, $\mcS$ the spectral category of $\mcA$, and $P:\mcA\to\mcS$ the canonical additive functor.
	\begin{enumerate}
		\item\label{item:essiso} Let $M$ be an object in $\mcA$ and $N$ an essential subobject of $M$. Denote the inclusion morphism by $\nu:N\hookrightarrow M$. Then $P(\nu)$ is an isomorphism in $\mcS$.
		\item Let $M$ and $N$ be objects in $\mcA$. If $P(M)$ is isomorphic to $P(N)$, then there exists an essential subobject of $M$ which is isomorphic to an essential subobject of $N$.
		\item For any object $M$ in $\mcA$, $M$ is uniform if and only if $P(M)$ is simple.
		\item\label{item:indecinjsimple} The correspondence $I\mapsto P(I)$ gives a bijection between the isomorphism classes of indecomposable injective objects in $\mcA$ and the isomorphism classes of simple objects in $\mcS$.
	\end{enumerate}
\end{Thm}

In the rest of this section, let $\mcA$ be a locally noetherian Grothendieck category, $\mcS$ the spectral category of $\mcA$, $P:\mcA\to\mcS$ the canonical additive functor, and $\alpha$ an atom in $\mcA$.

\begin{rem}\leavevmode
	\begin{enumerate}
		\item For any object $M$ in $\mcA$, as shown in \cite[Theorem 2.5]{Matlis}, $E(M)$ is a direct sum of indecomposable injective objects in $\mcA$. By Theorem \ref{Thm:spectralcat2} (\ref{item:essiso}), $P(M)$ is isomorphic to $P(E(M))$. Hence by Theorem \ref{Thm:spectralcat1} (\ref{item:leftexactdirectsums}) and Theorem \ref{Thm:spectralcat2} (\ref{item:indecinjsimple}), $P(M)$ is isomorphic to a direct sum of simple objects in $\mcS$.
		\item Since $E(\alpha)$ is an indecomposable injective object in $\mcA$, by Theorem \ref{Thm:spectralcat2} (\ref{item:indecinjsimple}), $P(E(\alpha))$ is a simple object in $\mcS$. Therefore $\End_{\mcS}(P(E(\alpha)))$ is a skew field.
	\end{enumerate}
\end{rem}

In terms of the spectral category, we define the morphism space between an atom and an object in $\mcA$. For a ring $R$, $\Mod R$ denotes the category of right $R$-modules.

\begin{Def}\label{Def:morspace}
	Let $\mcA$ be a locally noetherian Grothendieck category, $\mcS$ the spectral category of $\mcA$, $P:\mcA\to\mcS$ the canonical additive functor, and $\alpha$ an atom in $\mcA$.
	\begin{enumerate}
		\item Denote the skew field $\End_{\mcS}(P(E(\alpha)))$ by $k(\alpha)$ and call it the \emph{residue field} of $\alpha$.
		\item Define an additive functor $\Hom_{\mcA}(\alpha,-):\mcA\to\Mod k(\alpha)$ by
		\begin{equation*}
			\Hom_{\mcA}(\alpha,-)=\Hom_{\mcS}(P(E(\alpha)),P(-)).
		\end{equation*}
	\end{enumerate}
\end{Def}

\begin{rem}\label{rem:morspace}\leavevmode
	\begin{enumerate}
		\item\label{item:cofinal} For any nonzero subobject $U'$ of $E(\alpha)$, since $E(\alpha)$ is uniform, $\mcF_{U'}$ is the set of all the nonzero subobjects of $U'$, and it is a cofinal subset of $\mcF_{E(\alpha)}$. Hence we obtain a functorial isomorphism
	\begin{equation*}
		\Hom_{\mcA}(\alpha,-)\cong\varinjlim_{U\in\mcF_{U'}}\Hom_{\mcA}(U,-)
	\end{equation*}
	of additive functors $\mcA\to\Mod k(\alpha)$. If we take $U'$ as $H_{\alpha}$, then $\mcF_{U'}$ is the set of all the monoform subobjects of $E(\alpha)$. This is the reason why we denote this additive functor by $\Hom_{\mcA}(\alpha,-)$.
		\item By Theorem \ref{Thm:spectralcat1} (\ref{item:leftexactdirectsums}), the additive functor $\Hom_{\mcA}(\alpha,-)$ commutes with arbitrary direct sums since $P(E(\alpha))$ is a simple object in $\mcS$.
		\item\label{item:skewfield} For a nonzero subobject $U$ of $E(\alpha)$, an object $M$ in $\mcA$, and a morphism $f:U\to M$ in $\mcA$, we denote the image of $f$ in $\Hom_{\mcA}(\alpha,M)$ by $[f]$.
		
		In the case where $M$ is a nonzero subobject $U'$ of $E(\alpha)$, the composite of $f$ and the inclusion morphism $U'\hookrightarrow E(\alpha)$ defines an element of $k(\alpha)$. We also denote it by $[f]$ by abuse of notation. By Theorem \ref{Thm:spectralcat2} (\ref{item:essiso}), we have
		\begin{equation*}
			k(\alpha)\cong\Hom_{\mcS}(P(U),P(U')).
		\end{equation*}
		This implies that
		\begin{equation*}
			k(\alpha)=\{[f]\mid f:U\to U'\text{ in }\mcA,\ U\in\mcF_{E(\alpha)}\}.
		\end{equation*}
	\end{enumerate}
\end{rem}

If there exists a simple object $S$ in $\mcA$ such that $\alpha=\overline{S}$, by Remark \ref{rem:morspace} (\ref{item:cofinal}), we have an isomorphism $k(\alpha)\cong\End_{\mcA}(S)$ of skew fields and a functorial isomorphism
\begin{equation*}
	\Hom_{\mcA}(\alpha,-)\cong\Hom_{\mcA}(S,-)
\end{equation*}
of additive functors $\mcA\to\Mod k(\alpha)$. In general, however, the additive functor $\Hom_{\mcA}(\alpha,-):\mcA\to\Mod\mbZ$ is not necessarily representable.

\begin{Prop}
	The additive functor $\Hom_{\mcA}(\alpha,-):\mcA\to\Mod\mbZ$ is representable if and only if there exists a simple object $S$ in $\mcA$ such that $\alpha=\overline{S}$.
\end{Prop}

\begin{proof}
	Assume that $\Hom_{\mcA}(\alpha,-):\mcA\to\Mod\mbZ$ is representable, that is, there exists a functorial isomorphism
	\begin{equation*}
		\Hom_{\mcA}(\alpha,-)\cong\Hom_{\mcA}(U,-)
	\end{equation*}
	of additive functors $\mcA\to\Mod\mbZ$ for some object $U$ in $\mcA$. Then there exist a nonzero noetherian subobject $V$ of $U$ and a simple quotient object $S$ of $V$. We obtain a morphism $f$ in $\mcA$ such that the diagram
	\begin{equation*}
		\newdir^{ (}{!/-5pt/@^{(}}
		\xymatrix{
			U\ar^(0.4){f}[r] & E(S) \\
			V\ar@{^{ (}->}[u]\ar@{>>}[r] & S\ar@{^{ (}->}[u]
		}
	\end{equation*}
	commutes. Since $f$ is nonzero, we have
	\begin{equation*}
		\Hom_{\mcS}(P(E(\alpha)),P(E(S)))=\Hom_{\mcA}(\alpha,E(S))\cong\Hom_{\mcA}(U,E(S))\neq 0.
	\end{equation*}
	By Theorem \ref{Thm:spectralcat2} (\ref{item:indecinjsimple}), we have $E(\alpha)\cong E(S)\cong E(\overline{S})$. By Theorem \ref{Thm:atominj} (\ref{item:atomindecinj}), it follows that $\alpha=\overline{S}$.
\end{proof}

We show a property about the morphism space between an atom and an object.

\begin{Prop}\label{Prop:nonzeroiffmono}
	Let $U$ be a nonzero subobject of $E(\alpha)$, $M$ an object in $\mcA$, and $f:U\to M$ a morphism in $\mcA$. Then $[f]\neq 0$ in $\Hom_{\mcA}(\alpha,M)$ if and only if $f$ is a monomorphism.
\end{Prop}

\begin{proof}
	By the definition of direct limit, $[f]=0$ if and only if there exists a nonzero subobject $U'$ of $U$ such that $U'\subset \Ker f$. This condition is equivalent to that $f$ is not a monomorphism.
\end{proof}

The residue field $k(\alpha)$ of an atom $\alpha$ in $\mcA$ appears in several contexts related to the atom $\alpha$. In order to show that, we prove the following lemma.

\begin{Lem}\label{Lem:moratomicinj}
	Let $H\in\alpha$ and $f:H\to E(\alpha)$ be a morphism in $\mcA$. Then we have $\Im f\subset H_{\alpha}$.
\end{Lem}

\begin{proof}
	Since $E(\alpha)$ is isomorphic to $E(H)$, there exists a subobject $H'$ of $E(\alpha)$ which is isomorphic to $H$. Assume that $f$ is a nonzero morphism which is not a monomorphism. Then we have $\Ker f\neq 0$ and $\Im f\neq 0$, and hence, by the uniformness of $E(\alpha)$, we have $\Im f\cap H'\neq 0$. Then
	\begin{equation*}
		\frac{H}{\Ker f}\cong\Im f\supset\Im f\cap H'\subset H'\cong H
	\end{equation*}
	implies that $H$ is not monoform. This is a contradiction. Therefore $f$ is a monomorphism or a zero morphism, and hence $\Im f$ is monoform or zero. This implies that $\Im f\subset H_{\alpha}$.
\end{proof}

\begin{Prop}\label{Prop:residuefield}\leavevmode
	\begin{enumerate}
		\item\label{item:atomicendo} We have an isomorphism $\End_{\mcA}(H_{\alpha})\cong k(\alpha)$ of skew fields.
		\item\label{item:injendo} Let $J_{\alpha}$ be the unique maximal ideal of the local ring $\End_{\mcA}(E(\alpha))$. Then we have an isomorphism
		\begin{equation*}
			\frac{\End_{\mcA}(E(\alpha))}{J_{\alpha}}\cong k(\alpha)
		\end{equation*}
		of skew fields.
	\end{enumerate}
\end{Prop}

\begin{proof}
	(1) For any morphism $f:H_{\alpha}\to H_{\alpha}$ in $\mcA$, we have an element $[f]$ of $k(\alpha)$. This correspondence defines a ring homomorphism $\varphi:\End_{\mcA}(H_{\alpha})\to k(\alpha)$.
	
	If $f$ is nonzero, then we have $\Im f\neq 0$ and
	\begin{equation*}
		\frac{H_{\alpha}}{\Ker f}\cong\Im f\subset H_{\alpha}.
	\end{equation*}
	Since $H_{\alpha}$ is monoform, we have $\Ker f=0$. Hence $f$ is a monomorphism and by Proposition \ref{Prop:nonzeroiffmono}, we have $[f]\neq 0$. This shows that $\varphi$ is injective.
	
	In order to show that $\varphi$ is surjective, let $H$ be a monoform subobject of $E(\alpha)$ and $g:H\to E(\alpha)$ a morphism in $\mcA$ such that $[g]\neq 0$. By the injectivity of $E(\alpha)$, we obtain a morphism $g'$ in $\mcA$ such that the diagram
	\begin{equation*}
		\newdir^{ (}{!/-5pt/@^{(}}
		\xymatrix{
			H_{\alpha}\ar^{g'}[dr] & \\
			H\ar@{^{ (}->}[u]\ar^(0.4){g}[r] & E(\alpha)
		}
	\end{equation*}
	commutes. By Lemma \ref{Lem:moratomicinj}, we obtain a morphism $g''$ in $\mcA$ such that the diagram
	\begin{equation*}
		\newdir^{ (}{!/-5pt/@^{(}}
		\xymatrix{
			H_{\alpha}\ar^{g''}[r]\ar_{g'}[dr] & H_{\alpha}\ar@{^{ (}->}[d]\\
			 & E(\alpha)
		}
	\end{equation*}
	commutes. Then $[g'']=[g']=[g]$ in $k(\alpha)$. This shows that $\varphi$ is surjective.
	
	(2) The localness of $\End_{\mcA}(E(\alpha))$ is shown in \cite[Proposition 2.6]{Matlis}.
	
	We have a canonical ring homomorphism $\psi:\End_{\mcA}(E(\alpha))\to k(\alpha)$. For any nonzero subobject $U$ of $E(\alpha)$, the injectivity of $E(\alpha)$ ensures that every morphism $U\to E(\alpha)$ can be lifted to some endomorphism of $E(\alpha)$, and hence $\psi$ is surjective. The kernel of $\psi$ is the unique maximal two-sided ideal $J_{\alpha}$ since $k(\alpha)$ is a skew field.
\end{proof}

In the case where $\mcA=\Mod R$ for a commutative noetherian ring $R$, each prime ideal of $R$ defines an atom $\alpha=\overline{R/\mfp}$ in $\ASpec(\Mod R)$. As in \cite{Storrer}, the corresponding atomic object $H_{\alpha}$ is the residue field $k(\mfp)$ of $\mfp$. Hence the residue field $k(\alpha)\cong\End_{R}(k(\mfp))$ is also $k(\mfp)$. In general, however, for an atom $\alpha$ in $\mcA$, the atomic object $H_{\alpha}$ corresponding to $\alpha$ and the residue field $k(\alpha)$ of $\alpha$ do not necessarily coincide with each other.

\begin{ex}
	Let $R$ be the ring of $2\times 2$ lower triangular matrices
	\begin{equation*}
		\begin{bmatrix} K & 0 \\ K & K \end{bmatrix}
	\end{equation*}
	over a field $K$. The right $R$-module $S=[K\ 0]$ is simple, and hence monoform. The atomic object corresponding to $\alpha=\overline{S}$ is $H_{\alpha}=[K\ K]$. However, the residue field $k(\alpha)$ of $\alpha$ is isomorphic to $K$.
\end{ex}

In the rest of this section, we consider a relationship to associated atoms, which were introduced by Storrer \cite{Storrer} in the case of a module category. In the case of an abelian category, they are stated in \cite{Kanda}.

\begin{Def}\label{Def:assatom}
	Let $\mcA$ be a locally noetherian Grothendieck category. For an object $M$ in $\mcA$, a subset $\AAss M$ of $\ASpec\mcA$ is defined by
	\begin{equation*}
		\AAss M=\{\alpha\in\ASpec\mcA\mid\text{there exists }H\in\alpha\text{ which is a subobject of }M\}.
	\end{equation*}
	An element of $\AAss M$ is called an \emph{associated atom} of $M$.
\end{Def}

\begin{Prop}\label{Prop:assmor}
	For any object $M$ in $\mcA$ and any atom $\alpha$ in $\mcA$, $\alpha\in\AAss M$ if and only if $\Hom_{\mcA}(\alpha,M)\neq 0$.
\end{Prop}

\begin{proof}
	This follows from Proposition \ref{Prop:nonzeroiffmono}.
\end{proof}

We recall a fundamental result concerning associated atoms.

\begin{Prop}\label{Prop:assexact}
	Let $0\to L\to M\to N\to 0$ be an exact sequence in $\mcA$. Then we have
	\begin{equation*}
		\AAss L\subset\AAss M\subset\AAss L\cup\AAss N.
	\end{equation*}
\end{Prop}

\begin{proof}
	\cite[Proposition 3.5]{Kanda}.
\end{proof}

\section{Extension groups between atoms and objects}
\label{sec:extgroups}

Throughout this section, let $\mcA$ be a locally noetherian Grothendieck category, $\mcS$ the spectral category of $\mcA$, $P:\mcA\to\mcS$ the canonical additive functor, and $\alpha$ an atom in $\mcA$.

By Theorem \ref{Thm:spectralcat1} (\ref{item:spectralsplit}), the additive functor $\Hom_{\mcS}(P(E(\alpha)),-):\mcS\to\Mod k(\alpha)$ is exact. Since $P$ is a left exact functor, the additive functor $\Hom_{\mcA}(\alpha,-)=\Hom_{\mcS}(P(E(\alpha)),P(-)):\mcA\to\Mod k(\alpha)$ is left exact.

\begin{Def}\label{Def:extgroup}
	Let $\mcA$ be a locally noetherian Grothendieck category. For an atom $\alpha$ in $\mcA$, denote the right derived functor of the left exact functor $\Hom_{\mcA}(\alpha,-):\mcA\to\Mod k(\alpha)$ by $\RHom_{\mcA}(\alpha,-):\D(\mcA)\to\D(\Mod k(\alpha))$, where $\D(\mcA)$ is the unbounded derived category of $\mcA$. For a nonnegative integer $i$, the $i$-th right derived functor of $\Hom_{\mcA}(\alpha,-)$ is denoted by $\Ext^{i}_{\mcA}(\alpha,-):\mcA\to\Mod k(\alpha)$.
\end{Def}

The existence of the right derived functor follows from the fact that every complex in a Grothendieck category has a K-injective resolution (\cite{Spaltenstein} and \cite{AlonsoTarrioJeremiasLopezSoutoSalorio}).

In order to give other descriptions of the functors $\RHom_{\mcA}(\alpha,-)$ and $\Ext^{i}_{\mcA}(\alpha,-)$, we will see that the functor $\Hom_{\mcA}(\alpha,-)$ is the composite of additive functors $G_{\alpha}$ and $L_{\alpha}$ defined below.

\begin{Def}\leavevmode
	\begin{enumerate}
		\item Denote by $\mcC_{\alpha}$ the full subcategory of $\mcA$ consisting of objects which are isomorphic to nonzero subobjects of $E(\alpha)$.
		\item Denote by $\Mod\mcC_{\alpha}$ the category of contravariant $\mbZ$-functors from $\mcC_{\alpha}$ to $\Mod\mbZ$. (Note that the category $\mcC_{\alpha}$ is skeletally small since $\mcA$ is a Grothendieck category.)
	\end{enumerate}
\end{Def}

\begin{rem}\leavevmode
	\begin{enumerate}
		\item Regard the directed set $\mcF_{E(\alpha)}$ as a category (see section \ref{sec:morspaces}). Then we have a canonical contravariant functor $\mcF_{E(\alpha)}\to\mcC_{\alpha}$, which sends the unique morphism $U\to U'$ in $\mcF_{E(\alpha)}$ for each pair of subobjects $U\supset U'$ to the inclusion morphism $U'\hookrightarrow U$. This canonical functor is faithful and dense.
		\item $\Mod\mcC_{\alpha}$ is a Grothendieck category.
	\end{enumerate}
\end{rem}

\begin{Def}
	For a nonnegative integer $i$, define an additive functor $G^{i}_{\alpha}:\mcA\to\Mod\mcC_{\alpha}$ by $G^{i}_{\alpha}(M)=\Ext^{i}_{\mcA}(-,M)|_{\mcC_{\alpha}}$ for each object $M$ in $\mcA$ and $G^{i}_{\alpha}(f)=\Ext^{i}_{\mcA}(-,f)|_{\mcC_{\alpha}}$ for each morphism $f$ in $\mcA$, where $(-)|_{\mcC_{\alpha}}$ is the restriction to $\mcC_{\alpha}$. $G^{0}_{\alpha}$ is also denoted by $G_{\alpha}$.
\end{Def}

$G_{\alpha}$ is a left exact functor and its $i$-th right derived functor is $G^{i}_{\alpha}$.

\begin{Prop}\leavevmode
	\begin{enumerate}
		\item Let $X$ be an object in $\Mod\mcC_{\alpha}$. Then the composite of the canonical contravariant functor $\mcF_{E(\alpha)}\to\mcC_{\alpha}$ and the contravariant additive functor $X:\mcC_{\alpha}\to\Mod\mbZ$ defines a direct system in $\Mod\mbZ$. Denote its direct limit in $\Mod\mbZ$ by $L_{\alpha}(X)$.
		\item The correspondence $X\mapsto L_{\alpha}(X)$ defines an exact functor $L_{\alpha}:\Mod\mcC_{\alpha}\to\Mod k(\alpha)$.
	\end{enumerate}
\end{Prop}

\begin{proof}
	(1) This is obvious.
	
	(2) In order to show that $L_{\alpha}(X)$ has a canonical $k(\alpha)$-module structure, let $[x]\in L_{\alpha}(X)$ and $[f]\in k(\alpha)$, where $x\in X(U)$, $f:U'\to U$, and $U,U'\in\mcF_{E(\alpha)}$ (see Remark \ref{rem:morspace} (\ref{item:skewfield})). Define $[x][f]\in L_{\alpha}(X)$ by $[x][f]=[X(f)(x)]$.
	
	In order to show the well-definedness of this action, let $y\in X(V)$ and $g:V'\to V$ such that $[x]=[y]$, and $[f]=[g]$. In the case where $N$ is a subobject of an object $M$ in $\mcA$, denote the inclusion morphism $N\hookrightarrow M$ by $\iota_{N,M}$. By the definition of $L_{\alpha}(X)$, there exists $W\in\mcF_{E(\alpha)}$ such that $W\subset U\cap V$, and $X(\iota_{W,U})(x)=X(\iota_{W,V})(y)$. Since $[f]=[g]$, there exists $W'\in\mcF_{E(\alpha)}$ such that $W'\subset U'\cap V'$, and $\iota_{U,E(\alpha)}f\iota_{W',U'}=\iota_{V,E(\alpha)}g\iota_{W',V'}$. Denote this morphism by $r:W'\to E(\alpha)$. By replacing $W'$ by $W'\cap r^{-1}(W)$, we can assume $\Im r\subset W$. Then there exists a morphism $h:W'\to W$ such that $\iota_{W,U}h=f\iota_{W',U'}$, and $\iota_{W,V}h=g\iota_{W',V'}$. Hence we have
	\begin{align*}
		[X(f)(x)]
		&=[X(\iota_{W',U'})(X(f)(x))]=[X(f\iota_{W',U'})(x)]\\
		&=[X(\iota_{W,U}h)(x)]=[X(f)(X(\iota_{W,U})(x))]\\
		&=[X(f)(X(\iota_{W,V})(y))]=[X(g)(y)].
	\end{align*}
	This shows the well-definedness of $[x][f]$.
	
	It is straightforward to show that this action makes $L_{\alpha}(X)$ a $k(\alpha)$-module, and $L_{\alpha}$ becomes an additive functor. The exactness of $L_{\alpha}$ follows from that of direct limit.
\end{proof}

In section \ref{sec:morspaces}, we defined the left exact functor $\Hom_{\mcA}(\alpha,-)$ as the composite of the left exact functor $P:\mcA\to\mcS$ and the exact functor $\Hom_{\mcS}(P(E(\alpha)),-):\mcS\to\Mod k(\alpha)$. We also have the following description of $\Hom_{\mcA}(\alpha,-)$.

\begin{Prop}
	The left exact functor $\Hom_{\mcA}(\alpha,-)$ is the composite of the left exact functor $G_{\alpha}:\mcA\to\Mod\mcC_{\alpha}$ and the exact functor $L_{\alpha}:\Mod\mcC_{\alpha}\to\Mod k(\alpha)$.
\end{Prop}

\begin{proof}
	This can be shown straightforwardly.
\end{proof}

This observation allows us to give other descriptions of the functors $\RHom_{\mcA}(\alpha,-):\D(\mcA)\to\D(\Mod k(\alpha))$ and $\Ext^{i}_{\mcA}(\alpha,-):\mcA\to\Mod k(\alpha)$.

\begin{Thm}
	Let $\mcA$ be a locally noetherian Grothendieck category, $\mcS$ the spectral category of $\mcA$, $P:\mcA\to\mcS$ the canonical additive functor, and $\alpha$ an atom in $\mcA$.
	\begin{enumerate}
		\item There exist functorial isomorphisms
		\begin{equation*}
			\RHom_{\mcA}(\alpha,-)\cong\Hom_{\mcS}(P(E(\alpha)),-)\circ\R P\cong L_{\alpha}\circ\R G_{\alpha}
		\end{equation*}
		of triangle functors $\D(\mcA)\to\D(\Mod k(\alpha))$.
		\item For any nonnegative integer $i$, we have a functorial isomorphism
		\begin{equation*}
			\Ext^{i}_{\mcA}(\alpha,-)\cong\varinjlim_{U\in\mcF_{E(\alpha)}}\Ext^{i}_{\mcA}(U,-)
		\end{equation*}
		of additive functors $\mcA\to\Mod k(\alpha)$.
	\end{enumerate}
\end{Thm}

\begin{proof}
	(1) This follows from above observation.
	
	(2) For an object $M$ in $\mcA$, we have
	\begin{align*}
		\Ext^{i}_{\mcA}(\alpha,M)
		&=\H^{i}(\RHom_{\mcA}(\alpha,M))
		\cong\H^{i}(L_{\alpha}(\R G_{\alpha}(M)))\\
		&\cong L_{\alpha}(\H^{i}(\R G_{\alpha}(M)))
		=L_{\alpha}(\Ext^{i}_{\mcA}(-,M)|_{\mcC_{\alpha}})\\
		&=\varinjlim_{U\in\mcF_{E(\alpha)}}\Ext^{i}_{\mcA}(U,M).
	\end{align*}
	It is easy to see that these isomorphisms are functorial on $M$.
\end{proof}

\begin{rem}
	Similarly to Remark \ref{rem:morspace} (\ref{item:cofinal}), for any nonzero subobject $U'$ of $E(\alpha)$, we have a functorial isomorphism
	\begin{equation*}
		\Ext^{i}_{\mcA}(\alpha,-)\cong\varinjlim_{U\in\mcF_{U'}}\Ext^{i}_{\mcA}(U,-)
	\end{equation*}
	of additive functors $\mcA\to\Mod k(\alpha)$.
\end{rem}

\section{Bass numbers}
\label{sec:Bassnumbers}

Throughout this section, let $\mcA$ be a locally noetherian Grothendieck category. For an object $M$ in $\mcA$, denote by $E^{i}(M)$ the $i$-th term in the minimal injective resolution of $M$, that is,
\begin{equation*}
	0\to M\to E^{0}(M)\to E^{1}(M)\to E^{2}(M)\to\cdots
\end{equation*}
is the minimal injective resolution of $M$. Note that $E^{0}(M)=E(M)$.

By considering Theorem \ref{Thm:atominj} (\ref{item:atomindecinj}), we can define Bass numbers from the viewpoint of atoms as a generalization of Bass numbers defined by Bass \cite{Bass63}.

\begin{Def}\label{Def:Bassnumber}
	Let $\mcA$ be a locally noetherian Grothendieck category, $\alpha$ an atom in $\mcA$, $M$ an object in $\mcA$, and $i$ a nonnegative integer. Define the $i$\emph{-th Bass number} $\mu_{i}(\alpha,M)$ of $M$ with respect to $\alpha$ as the cardinal number satisfying
	\begin{equation*}
		E^{i}(M)=\bigoplus_{\alpha\in\ASpec\mcA}E(\alpha)^{\oplus\mu_{i}(\alpha,M)}.
	\end{equation*}
\end{Def}

\begin{rem}\leavevmode
	\begin{enumerate}
		\item In Definition \ref{Def:Bassnumber}, the existence of indecomposable decompositions of injective objects is shown by Matlis \cite[Theorem 2.5]{Matlis}. Since the endomorphism ring of any indecomposable injective object is a local ring, Krull-Remak-Schmidt-Azumaya's theorem \cite[Theorem 1]{Azumaya} ensures the uniqueness of indecomposable decompositions of injective objects. The uniqueness also can be shown by using Theorem \ref{Thm:Bassext}.
		\item In the case where $i=0$ in Definition \ref{Def:Bassnumber}, $\mu_{0}(\alpha,M)$ coincides with $r_{I_{\alpha}}(M)$ with the notation in \cite{GabrielOberst}.
	\end{enumerate}
\end{rem}

We show that Bass numbers defined above coincide with the dimensions of the extension groups between an atom and an object.

\begin{Thm}\label{Thm:Bassext}
	Let $\mcA$ be a locally noetherian Grothendieck category, $\alpha$ an atom in $\mcA$, $M$ an object in $\mcA$, and $i$ a nonnegative integer. Then we have the equation
	\begin{equation*}
		\mu_{i}(\alpha,M)=\dim_{k(\alpha)}\Ext^{i}_{\mcA}(\alpha,M).
	\end{equation*}
\end{Thm}

\begin{proof}
	For a nonnegative integer $j$, denote by $\mho^{j}(M)$ the $j$-th cosyzygy of an object $M$ in $\mcA$, that is, $\mho^{0}(M)=M$, and $\mho^{j+1}(M)$ is the cokernel of the inclusion morphism $\mho^{j}(M)\hookrightarrow E^{j}(M)$. Then we have the long exact sequence
	\begin{equation*}
		\xymatrix{
			0\ar[r] & \Hom_{\mcA}(\alpha,\mho^{j}(M))\ar[r] & \Hom_{\mcA}(\alpha,E^{j}(M))\ar[r] & \Hom_{\mcA}(\alpha,\mho^{j+1}(M))\ar`r[d]`[l]`[dlll]`[dll] [dll]\\
			& \Ext^{1}_{\mcA}(\alpha,\mho^{j}(M))\ar[r] & \Ext^{1}_{\mcA}(\alpha,E^{j}(M))\ar[r] & \Ext^{1}_{\mcA}(\alpha,\mho^{j+1}(M))\ar`r[d]`[l]`[dlll]`[dll] [dll] \\
			& \Ext^{2}_{\mcA}(\alpha,\mho^{j}(M))\ar[r] & \Ext^{2}_{\mcA}(\alpha,E^{j}(M))\ar[r] & \Ext^{2}_{\mcA}(\alpha,\mho^{j+1}(M))\ar`r[d]`[l]`[dlll]`/8pt[dll] [dll] \\
			& \cdots & &
		}
	\end{equation*}
	in $\Mod k(\alpha)$. Since $\mho^{j}(M)$ is an essential subobject of $E^{j}(M)$, by Theorem \ref{Thm:spectralcat2} (\ref{item:essiso}), the morphism $\Hom_{\mcA}(\alpha,\mho^{j}(M))\to\Hom_{\mcA}(\alpha,E^{j}(M))$ is an isomorphism. For a positive integer $l$, we have $\Ext^{l}_{\mcA}(\alpha,E^{j}(M))=0$. Hence we obtain $\Ext^{l}_{\mcA}(\alpha,\mho^{j+1}(M))\cong\Ext^{l+1}_{\mcA}(\alpha,\mho^{j}(M))$ for any nonnegative integer $l$. By using this isomorphism repeatedly, we obtain
	\begin{align*}
		\Ext^{i}_{\mcA}(\alpha,M)
		&\cong\Hom_{\mcA}(\alpha,\mho^{i}(M))\\
		&\cong\Hom_{\mcA}(\alpha,E^{i}(M))\\
		&=\Hom_{\mcA}(\alpha,\bigoplus_{\beta\in\ASpec\mcA}E(\beta)^{\oplus\mu_{i}(\beta,M)})\\
		&\cong\bigoplus_{\beta\in\ASpec\mcA}\Hom_{\mcA}(\alpha,E(\beta))^{\oplus\mu_{i}(\beta,M)}\\
		&=k(\alpha)^{\oplus\mu_{i}(\alpha,M)}.
	\end{align*}
	Therefore the statement holds.
\end{proof}

\section{Application to E-stable subcategories}
\label{sec:application}

Throughout this section, let $\mcA$ be a locally noetherian Grothendieck category. As an application of Theorem \ref{Thm:Bassext}, we show some properties about E-stable subcategories. We recall their definition.

\begin{Def}
	Let $\mcA$ be a locally noetherian Grothendieck category and $\mcX$ a full subcategory of $\mcA$.
	\begin{enumerate}
		\item $\mcX$ is called $E$\emph{-stable} if for any object $M$ in $\mcA$, $M$ belongs to $\mcX$ if and only if $E^{i}(M)$ belongs to $\mcX$ for each nonnegative integer $i$.
		\item We say that $\mcX$ is \emph{closed under arbitrary direct sums} if for any family $\{M_{\lambda}\}_{\lambda\in\Lambda}$ of objects in $\mcX$, $\bigoplus_{\lambda\in\Lambda}M_{\lambda}$ also belongs to $\mcX$.
		\item We say that $\mcX$ is \emph{closed under direct summands} if for any object $M$ in $\mcX$, any direct summand of $M$ also belongs to $\mcX$.
	\end{enumerate}
\end{Def}

We define the small atom support of an object. This generalizes the notion of the small support in the setting of commutative noetherian rings (see \cite[Remark 2.9]{Foxby}).

\begin{Def}
	Let $\mcA$ be a locally noetherian Grothendieck category.
	\begin{enumerate}
		\item For an object $M$ in $\mcA$, define a subset $\asupp M$ of $\ASpec\mcA$ by
		\begin{equation*}
			\asupp M=\{\alpha\in\ASpec\mcA\mid\mu_{i}(\alpha,M)\neq 0\text{ for some }i\in\mbZ_{\geq 0}\},
		\end{equation*}
		and call it the \emph{small atom support} of $M$.
		\item For a full subcategory $\mcX$ of $\mcA$, define a subset $\asupp\mcX$ of $\ASpec\mcA$ by
		\begin{equation*}
			\asupp\mcX=\bigcup_{M\in\mcX}\asupp M.
		\end{equation*}
		\item For a subset $\Phi$ of $\ASpec\mcA$, define a full subcategory $\asupp^{-1}\Phi$ of $\mcA$ by
		\begin{equation*}
			\asupp^{-1}\Phi=\{M\in\mcA\mid\asupp M\subset\Phi\}.
		\end{equation*}
	\end{enumerate}
\end{Def}

The small atom support is related to associated atoms in Definition \ref{Def:assatom} as follows.

\begin{Prop}\label{Prop:asupp}
	Let $M$ be an object in $\mcA$. Then the following hold.
	\begin{enumerate}
		\item\label{item:asupprhom} $\asupp M=\{\alpha\in\ASpec\mcA\mid\RHom_{\mcA}(\alpha,M)\neq 0\}$.
		\item $\AAss M\subset \asupp M$.
	\end{enumerate}
\end{Prop}

\begin{proof}
	(1) This follows from Theorem \ref{Thm:Bassext}.
	
	(2) This follows from (\ref{item:asupprhom}) and Proposition \ref{Prop:assmor}.
\end{proof}

\begin{Thm}\label{Thm:subcatatom}
	Let $\mcA$ be a locally noetherian Grothendieck category. Then the map $\mcX\mapsto\asupp\mcX$ is a bijection between the E-stable subcategories of $\mcA$ closed under arbitrary direct sums and direct summands, and the subsets of $\ASpec\mcA$. The inverse map is given by $\Phi\mapsto\asupp^{-1}\Phi$.
\end{Thm}

\begin{proof}
	Let $\Phi$ be a subset of $\ASpec\mcA$. Then $\asupp^{-1}\Phi$ is E-stable and closed under direct summands. It is also closed under arbitrary direct sums since for any family $\{M_{\lambda}\}_{\lambda\in\Lambda}$ of objects in $\mcA$, the minimal injective resolution of $\bigoplus_{\lambda\in\Lambda}M_{\lambda}$ is
	\begin{equation*}
		0\to\bigoplus_{\lambda\in\Lambda}M_{\lambda}\to\bigoplus_{\lambda\in\Lambda}E^{0}(M_{\lambda})\to\bigoplus_{\lambda\in\Lambda}E^{1}(M_{\lambda})\to\bigoplus_{\lambda\in\Lambda}E^{2}(M_{\lambda})\to\cdots
	\end{equation*}
	by \cite[Proposition 2.1]{Matlis}. The inclusion $\asupp\asupp^{-1}\Phi\subset\Phi$ is obvious. For any $\alpha\in\Phi$, since $E(\alpha)$ belongs to $\asupp^{-1}\Phi$, we have $\alpha\in\asupp E(\alpha)\subset\asupp\asupp^{-1}\Phi$. Hence $\asupp\asupp^{-1}\Phi=\Phi$.
	
	Let $\mcX$ be an E-stable subcategory of $\mcA$ which is closed under arbitrary direct sums and direct summands. Then an object $M$ in $\mcA$ belongs to $\mcX$ if and only if $E(\alpha)$ belongs to $\mcX$ for any $\alpha\in\asupp M$. For any atom $\alpha$ in $\mcA$, if $E(\alpha)$ belongs to $\mcX$, then $\alpha\in\asupp E(\alpha)\subset\asupp\mcX$. Conversely, if $\alpha\in\asupp\mcX$, there exist an object $N$ in $\mcX$ and a nonnegative integer $i$, $E(\alpha)$ is a direct summand of $E^{i}(N)$, and hence $E(\alpha)$ belongs to $\mcX$. Therefore $M$ belongs to $\mcX$ if and only if $\asupp M\subset\asupp\mcX$. This implies that $\mcX=\asupp^{-1}\asupp\mcX$.
\end{proof}

By using the long exact sequence of extension groups between atoms and objects, we can show the following result, which is a generalization of \cite[Corollary 2.19]{Takahashi}.

\begin{Cor}\label{Cor:subcatexact}
	Let $\mcA$ be a locally noetherian Grothendieck category and $\mcX$ an E-stable subcategories of $\mcA$ closed under arbitrary direct sums and direct summands. For any exact sequence
	\begin{equation*}
		0\to L\to M\to N\to 0
	\end{equation*}
	in $\mcA$, if two of $L$, $M$, $N$ belong to $\mcX$, then the remaining one also belongs to $\mcX$.
\end{Cor}

\begin{proof}
	This short exact sequence induces a triangle
	\begin{equation*}
		L\to M\to N\to L[1]
	\end{equation*}
	in $\D(\mcA)$. For any atom $\alpha$ in $\mcA$, we have the triangle
	\begin{equation*}
		\RHom_{\mcA}(\alpha,L)\to\RHom_{\mcA}(\alpha,M)\to\RHom_{\mcA}(\alpha,N)\to\RHom_{\mcA}(\alpha,L)[1]
	\end{equation*}
	in $\D(\Mod k(\alpha))$. Then by Proposition \ref{Prop:asupp} (\ref{item:asupprhom}), we have
	\begin{align*}
		\asupp L&\subset\asupp M\cup\asupp N,\\
		\asupp M&\subset\asupp L\cup\asupp N,\\
		\asupp N&\subset\asupp L\cup\asupp M.
	\end{align*}
	Then the claim follows. For example, if $M$ and $N$ belong to $\mcX$, then by Theorem \ref{Thm:subcatatom}, we have $\asupp M\subset\asupp\mcX$ and $\asupp N\subset\asupp\mcX$. Therefore we deduce that $\asupp L\subset\asupp\mcX$, and hence $L$ belongs to $\mcX$.
\end{proof}

\section{Bass numbers for noetherian algebras}
\label{sec:modfinalg}

In this section, we give another description of Bass numbers in the case of noetherian algebras. Throughout this section, let $R$ be a commutative noetherian ring and $\Lambda$ a ring whose center contains $R$ as a subring, and assume that $\Lambda$ is finitely generated as an $R$-module. A module means a right module, and an ideal means a two-sided ideal.

A proper ideal $P$ of $\Lambda$ is called a \emph{prime ideal} of $\Lambda$ if for any $a,b\in \Lambda$, $a\Lambda b\subset P$ implies $a\in P$ or $b\in P$. The set of all the prime ideals of $\Lambda$ is denoted by $\Spec\Lambda$, and the set of all the maximal ideals of $\Lambda$ is denoted by $\Max\Lambda$. Note that $\Max\Lambda\subset\Spec\Lambda$. For a $\Lambda$-module $M$, denote by $\Ass_{\Lambda}M$ the set of all the associated prime ideals of $M$, that is, a prime ideal $P$ of $\Lambda$ belongs to $\Ass_{\Lambda}M$ if and only if there exists a nonzero $\Lambda$-submodule $N$ of $M$ such that for any nonzero $\Lambda$-submodule $N'$ of $N$, $\Ann_{\Lambda}(N')=P$.

On indecomposable injective $\Lambda$-modules, the following fact is known.

\begin{Thm}\label{Thm:primeindecinj}\leavevmode
	\begin{enumerate}
		\item\label{item:Thmprimeindecinj} {\rm (\cite[V, 4, Lemma 2]{Gabriel})} There exists a bijection between $\Spec\Lambda$ and the set of isomorphism classes of indecomposable injective $\Lambda$-modules. For each prime ideal $P$ of $\Lambda$, denote by $I(P)$ the corresponding indecomposable injective $\Lambda$-module. Then the injective envelope $E_{\Lambda}(\Lambda/P)$ is the direct sum of finitely many copies of $I(P)$. We also have $\Ass_{\Lambda}I(P)=\Ass_{\Lambda}(\Lambda/P)=\{P\}$.
		\item\label{item:ThmAssindecinj} {\rm (\cite[V, 4, Proposition 6]{Gabriel})} Let $M$ be a $\Lambda$-module and $P$ a prime ideal of $\Lambda$. Then $P\in\Ass_{\Lambda}M$ if and only if $I(P)$ is a direct summand of $E_{\Lambda}(M)$.
	\end{enumerate}
\end{Thm}

Therefore we have a description of the atom spectrum of $\Mod\Lambda$.

\begin{Thm}\label{Thm:primeatom}\leavevmode
	\begin{enumerate}
		\item\label{item:Thmprimeatom} There exists a bijection between $\Spec\Lambda$ and $\ASpec(\Mod\Lambda)$. For each prime ideal $P$ of $\Lambda$, the corresponding atom $\widetilde{P}$ is determined by $\AAss(\Lambda/P)=\{\widetilde{P}\}$.
		\item\label{item:ThmAssAAss} The bijection in (\ref{item:Thmprimeatom}) induces a bijection between $\Ass_{\Lambda}M$ and $\AAss M$ for any $\Lambda$-module $M$.
	\end{enumerate}
\end{Thm}

\begin{proof}
	(\ref{item:Thmprimeatom}) This follows from Theorem \ref{Thm:primeindecinj} (\ref{item:Thmprimeindecinj}) and Theorem \ref{Thm:atominj} (\ref{item:atomindecinj}).
	
	(\ref{item:ThmAssAAss}) Let $P$ be a prime ideal of $\Lambda$. By Theorem \ref{Thm:primeindecinj} (\ref{item:ThmAssAAss}), we have $P\in\Ass_{\Lambda}M$ if and only if $\mu_{0}(\widetilde{P},M)\neq 0$. By Theorem \ref{Thm:Bassext} and Proposition \ref{Prop:assmor}, this is equivalent to $\widetilde{P}\in\AAss M$.
\end{proof}

The following lemma is useful to see behavior of $\Lambda/P$ for a prime ideal $P$ of $\Lambda$.

\begin{Lem}[{Goto and Nishida \cite[Lemma 2.5.1]{GotoNishida}}]\label{Lem:primeembedding}
	Let $P$ be a prime ideal of $\Lambda$ and $M$ a $\Lambda$-module. Then $P\in\Ass_{\Lambda}M$ if and only if there exist a positive integer $i$ and a $\Lambda$-monomorphism $\Lambda/P\hookrightarrow M^{\oplus i}$.
\end{Lem}

\begin{Prop}\label{Prop:maxsimple}
	The correspondence $P\mapsto\widetilde{P}$ gives a bijection between $\Max\Lambda$ and the set of atoms which are represented by simple $\Lambda$-modules. Therefore $\Max\Lambda$ bijectively corresponds to the set of all the isomorphism classes of simple $\Lambda$-modules.
\end{Prop}

\begin{proof}
	Let $P$ be a maximal ideal of $\Lambda$. Since $\Lambda/P$ is a noetherian $\Lambda$-module, there exist a simple $\Lambda$-module $S$ and a $\Lambda$-epimorphism $\Lambda/P\twoheadrightarrow S$. Since $P\subset\Ann_{\Lambda}(S)\subsetneq\Lambda$, we have $P=\Ann_{\Lambda}(S)$ by the maximality of $P$. Hence $\Ass_{\Lambda}S=\{P\}$, and by Theorem \ref{Thm:primeatom} (\ref{item:ThmAssAAss}), $\{\widetilde{P}\}=\AAss S=\{\overline{S}\}$.
	
	Conversely, let $P$ be a prime ideal of $\Lambda$, and assume that $\widetilde{P}$ is represented by a simple $\Lambda$-module $S$. Since $\AAss S=\{\overline{S}\}=\{\widetilde{P}\}$, by Theorem \ref{Thm:primeatom} (\ref{item:ThmAssAAss}), we have $\Ass_{\Lambda}S=\{P\}$. By Lemma \ref{Lem:primeembedding}, there exist a positive integer $i$ and a $\Lambda$-monomorphism $\Lambda/P\hookrightarrow S^{\oplus i}$, and hence $\Lambda/P$ is isomorphic to $S^{\oplus j}$ for some positive integer $j$. Let $Q$ be a maximal ideal of $\Lambda$ such that $P\subset Q$. Since $\Lambda/Q$ is a quotient $\Lambda$-module of $\Lambda/P\cong S^{\oplus j}$, $\Lambda/Q$ is isomorphic to $S^{\oplus l}$ for some positive integer $l$. By Proposition \ref{Prop:assexact}, $\{\widetilde{P}\}=\AAss(\Lambda/P)=\AAss S=\AAss(\Lambda/Q)=\{\widetilde{Q}\}$, and hence $P=Q$ is a maximal ideal of $\Lambda$.
\end{proof}

For a prime ideal $\mfp$ of $R$, the $R_{\mfp}$-algebra $\Lambda_{\mfp}$ also satisfies our assumption. We recall a description of prime ideals of $\Lambda_{\mfp}$.

\begin{Prop}\label{Prop:localprimemax}
	Let $\mfp$ be a prime ideal of $R$.
	\begin{enumerate}
		\item\label{item:localprime} The map
		\begin{equation*}
			\{Q\in\Spec\Lambda\mid Q\cap R\subset\mfp\}\to\Spec\Lambda_{\mfp},\ Q\mapsto Q\Lambda_{\mfp}
		\end{equation*}
		is bijective. The inverse map is given by $Q'\mapsto \varphi^{-1}(Q')$, where $\varphi:\Lambda\to\Lambda_{\mfp}$ is the canonical ring homomorphism.
		\item\label{item:localmax} The map in (\ref{item:localprime}) induces a bijection
		\begin{equation*}
			\{P\in\Spec\Lambda\mid P\cap R=\mfp\}\to\Max\Lambda_{\mfp}.
		\end{equation*}
	\end{enumerate}
\end{Prop}

\begin{proof}
	(1) This can be shown straightforwardly.
	
	(2) Let $Q$ be a prime ideal of $\Lambda$ such that $Q\cap R\subset\mfp$. By \cite[10.2.13, Proposition (iii)]{McConnellRobson}, $Q\Lambda_{\mfp}$ is a maximal ideal of $\Lambda_{\mfp}$ if and only if $Q\Lambda_{\mfp}\cap R_{\mfp}$ is a maximal ideal of $R_{\mfp}$. Since $Q\Lambda_{\mfp}\cap R_{\mfp}=(Q\cap R)R_{\mfp}$, this condition is equivalent to $Q\cap R=\mfp$.
\end{proof}

For any prime ideal $P$ of $\Lambda$, it can be easily shown that $\mfp=P\cap R$ is a prime ideal of $R$. Then by Proposition \ref{Prop:localprimemax} (\ref{item:localmax}), $P\Lambda_{\mfp}$ is a maximal ideal of $\Lambda_{\mfp}$. By Proposition \ref{Prop:maxsimple}, $P\Lambda_{\mfp}$ determines an isomorphism class $S(P)$ of simple $\Lambda_{\mfp}$-modules. We also denote by $S(P)$ a simple $\Lambda_{\mfp}$-module which represents the isomorphism class $S(P)$.

\begin{Thm}\label{Thm:primesimple}
	The map
	\begin{equation*}
		\Spec\Lambda\to\coprod_{\mfp\in\Spec R}\mcS_{\mfp},\ P\mapsto S(P)
	\end{equation*}
	is bijective, where for a prime ideal $\mfp$ of $R$, $\mcS_{\mfp}$ is the set of all the isomorphism classes of simple $\Lambda_{\mfp}$-modules.
	
	Consequently, the map
	\begin{equation*}
		\ASpec(\Mod\Lambda)\to\coprod_{\mfp\in\Spec R}\mcS_{\mfp},\ \widetilde{P}\mapsto S(P)
	\end{equation*}
	is bijective.
\end{Thm}

\begin{proof}
	The first assertion follows from Proposition \ref{Prop:localprimemax} (\ref{item:localmax}) and the definition of $S(P)$. By Theorem \ref{Thm:primeatom} (\ref{item:Thmprimeatom}), we obtain the description of the atom spectrum.
\end{proof}

The following lemma clarifies injective modules over $\Lambda_{\mfp}$ for each prime ideal $\mfp$ of $R$.

\begin{Lem}\label{Lem:localinj}
	Let $I$ be an injective $\Lambda$-module and $\mfp$ a prime ideal of $R$.
	\begin{enumerate}
		\item\label{item:injlocal} {\rm (Bass \cite[Lemma 1.2]{Bass62})} $I_{\mfp}$ is an injective $\Lambda_{\mfp}$-module.
		\item\label{item:injlocalsplit} {\rm (Goto and Nishida \cite[Lemma 2.4.1]{GotoNishida})} The canonical $\Lambda$-homomorphism $I\to I_{\mfp}$ is a split $\Lambda$-epimorphism.
	\end{enumerate}
\end{Lem}

The residue field $k(\widetilde{P})$ of the atom $\widetilde{P}$ can be described by using $S(P)$.

\begin{Prop}\label{Prop:residuefieldsimple}
	Let $P$ be a prime ideal of $\Lambda$ and $\mfp=P\cap R$. Then there exists an isomorphism $k(\widetilde{P})\cong\End_{\Lambda_{\mfp}}(S(P))$ of skew fields.
\end{Prop}

\begin{proof}
	Set $I=I(P)$. By Theorem \ref{Thm:primeindecinj} (\ref{item:Thmprimeindecinj}), $\Ass_{\Lambda}I=\{P\}$. Hence by Lemma \ref{Lem:primeembedding}, there exist a positive integer $i$ and a $\Lambda$-monomorphism $\Lambda/P\hookrightarrow I^{\oplus i}$. By localizing this morphism at $\mfp$, we have $\Lambda_{\mfp}/P\Lambda_{\mfp}\hookrightarrow I_{\mfp}^{\oplus i}$. By Proposition \ref{Prop:maxsimple} and the proof of it, $\Lambda_{\mfp}/P\Lambda_{\mfp}\cong S(P)^{\oplus j}$ in $\Mod\Lambda_{\mfp}$ for some positive integer $j$. By Lemma \ref{Lem:localinj}, $I_{\mfp}$ is an indecomposable injective $\Lambda_{\mfp}$-module, and $I_{\mfp}\cong I$ in $\Mod\Lambda$. Since $I_{\mfp}$ contains $S(P)$ as a $\Lambda_{\mfp}$-submodule, we have $I_{\mfp}\cong E_{\Lambda_{\mfp}}(S(P))$. Since any $\Lambda_{\mfp}$-homomorphism $S(P)\to S(P)$ can be lifted to a $\Lambda_{\mfp}$-homomorphism $I_{\mfp}\to I_{\mfp}$, the ring homomorphism $\End_{\Lambda_{\mfp}}(I_{\mfp})\to\End_{\Lambda_{\mfp}}(S(P))$ is surjective. Then we have a surjective ring homomorphism $\End_{\Lambda}(I)\cong\End_{\Lambda}(I_{\mfp})\cong\End_{\Lambda_{\mfp}}(I_{\mfp})\to\End_{\Lambda_{\mfp}}(S(P))$. By Proposition \ref{Prop:residuefield} (\ref{item:injendo}), $\End_{\Lambda_{\mfp}}(S(P))$ is isomorphic to $k(\widetilde{P})$.
\end{proof}

In order to give a description of Bass numbers, we regard $S(P)$ as a $\Lambda_{\mfp}$-submodule of $I(P)$ for each prime ideal $P$ of $\Lambda$. In fact, $S(P)$ is characterized as follows.

\begin{Lem}
	Let $P$ be a prime ideal of $\Lambda$ and $\mfp=P\cap R$. Then $\{x\in I(P)\mid xP=0\}$ is a $\Lambda_{\mfp}$-module which is isomorphic to $S(P)$.
\end{Lem}

\begin{proof}
	By Lemma \ref{Lem:localinj} (\ref{item:injlocalsplit}), $I(P)$ can be regarded as an indecomposable injective $\Lambda_{\mfp}$-module. As in the proof of Proposition \ref{Prop:residuefieldsimple}, $S(P)$ is isomorphic to a unique simple $\Lambda_{\mfp}$-submodule of $I(P)$. Set $N=\{x\in I(P)\mid xP=0\}$. For any $x\in N$ and $s\in R\setminus\mfp$, set $y=xs^{-1}$. Since $yPs=xP=0$, and $s$ acts on $I(P)$ as an isomorphism, we obtain $yP=0$. Hence $xs^{-1}\in N$, and this shows that $N$ is a $\Lambda_{\mfp}$-submodule of $I(P)$. Since $I(P)$ is a uniform $\Lambda_{\mfp}$-module, and $NP\Lambda_{\mfp}=0$, $N$ is a uniform $\Lambda_{\mfp}/P\Lambda_{\mfp}$-module. Since $\Lambda_{\mfp}/P\Lambda_{\mfp}$ is a finite-dimensional $R_{\mfp}/\mfp R_{\mfp}$-algebra, it is artinian. Furthermore $\Lambda_{\mfp}/P\Lambda_{\mfp}$ is a simple ring and hence is Morita-equivalent to a skew field by Artin-Wedderburn's theorem. Therefore $N$ is a simple $\Lambda_{\mfp}/P\Lambda_{\mfp}$-module and hence is also simple as a $\Lambda_{\mfp}$-module.
\end{proof}

\begin{Thm}\label{Thm:extBasssimple}
	Let $P$ be a prime ideal of $\Lambda$ and $\mfp=P\cap R$.
	\begin{enumerate}
		\item\label{item:extsimple} There exists a functorial isomorphism
		\begin{equation*}
			\RHom_{\Lambda}(\widetilde{P},-)\cong\RHom_{\Lambda_{\mfp}}(S(P),(-)_{\mfp})
		\end{equation*}
		of triangle functors $\D(\Mod\Lambda)\to\D(\Mod k(\widetilde{P}))$.
		\item For any nonnegative integer $i$ and any $\Lambda$-module $M$,
		\begin{equation*}
			\mu_{i}(\widetilde{P},M)=\dim_{k(\widetilde{P})}\Ext^{i}_{\Lambda_{\mfp}}(S(P),M_{\mfp}).
		\end{equation*}
	\end{enumerate}
\end{Thm}

\begin{proof}
	(1) By \cite[Proposition 2.12]{BokstedtNeeman}, every complex in $\Mod\Lambda$ has a K-projective resolution. Since the localization functor $(-)_{\mfp}=-\otimes\Lambda_{\mfp}:\Mod\Lambda\to\Mod\Lambda_{\mfp}$ has an exact right adjoint $\Hom_{\Lambda_{\mfp}}(\Lambda_{\mfp},-):\Mod\Lambda_{\mfp}\to\Mod\Lambda$, it sends any K-projective complex in $\Mod\Lambda$ to a K-projective complex in $\Mod\Lambda_{\mfp}$. Hence by \cite[Proposition 5.4]{Hartshorne}, the right derived functor of $\Hom_{\Lambda_{\mfp}}(S(P),(-)_{\mfp})$ is the composite of the induced triangle functor $(-)_{\mfp}:\D(\Mod\Lambda)\to\D(\Mod\Lambda_{\mfp})$ and the right derived functor $\RHom_{\Lambda_{\mfp}}(S(P),-):\D(\Mod\Lambda_{\mfp})\to\D(\Mod k(\widetilde{P}))$.
	Therefore it suffices to show that there exists a functorial isomorphism
	\begin{equation*}
		\Hom_{\Lambda}(\widetilde{P},-)\cong\Hom_{\Lambda_{\mfp}}(S(P),(-)_{\mfp})
	\end{equation*}
	of additive functors $\Mod\Lambda\to\Mod k(\widetilde{P})$.
	
	We regard $S(P)$ as a $\Lambda_{\mfp}$-submodule of $I(P)$. By definition, for any $\Lambda$-module $M$,
	\begin{equation*}
		\Hom_{\Lambda}(\widetilde{P},M)=\varinjlim_{U\in\mcF_{I(P)}}\Hom_{\Lambda}(U,M),
	\end{equation*}
	where $\mcF_{I(P)}$ is the directed set of all the nonzero $\Lambda$-submodules of $I(P)$. For any $U\in\mcF_{I(P)}$ and any $\Lambda$-homomorphism $f:U\to M$, we have $0\neq U\hookrightarrow U_{\mfp}\hookrightarrow I(P)_{\mfp}=I(P)$ by Lemma \ref{Lem:localinj} (\ref{item:injlocalsplit}), and hence the composite of $S(P)\hookrightarrow U_{\mfp}$ and $f_{\mfp}:U_{\mfp}\to M_{\mfp}$ is an element of $\Hom_{\Lambda_{\mfp}}(S(P),M_{\mfp})$. This correspondence defines a $\mbZ$-homomorphism $\varphi_{M}:\Hom_{\Lambda}(\widetilde{P},M)\to\Hom_{\Lambda_{\mfp}}(S(P),M_{\mfp})$.
		
	In order to show that $\varphi_{M}$ is a $k(\widetilde{P})$-homomorphism, let $[h]\in k(\widetilde{P})$, where $h:V\to U$, and $V\in\mcF_{I(P)}$. Then there exists a $\Lambda$-homomorphism $\widetilde{h}:I(P)\to I(P)$ such that the diagram
	\begin{equation*}
		\newdir^{ (}{!/-5pt/@^{(}}
		\xymatrix{
			I(P)\ar[r]^{\widetilde{h}} & I(P) \\
			V\ar@{^{ (}->}[u]\ar[r]^{h} & U\ar@{^{ (}->}[u]
		}
	\end{equation*}
	commutes. By Proposition \ref{Prop:residuefieldsimple}, $[h]\in k(\widetilde{P})$ corresponds to a $\Lambda_{\mfp}$-homomorphism $\theta:S(P)\to S(P)$. The commutative diagram
	\begin{equation*}
		\newdir^{ (}{!/-5pt/@^{(}}
		\xymatrix{
			I(P)\ar[r]^{\widetilde{h}} & I(P) & \\
			V_{\mfp}\ar@{^{ (}->}[u]\ar[r]^{h_{\mfp}} & U_{\mfp}\ar@{^{ (}->}[u]\ar[r]^{f_{\mfp}} & M_{\mfp}\\
			S(P)\ar@{^{ (}->}[u]\ar[r]^{\theta} & S(P)\ar@{^{ (}->}[u] &
		}
	\end{equation*}
	shows that $\varphi_{M}$ is a $k(\widetilde{P})$-homomorphism.
	
	Assume that $\varphi_{M}([f])=0$. Then by the definition of $\varphi_{M}$, $f_{\mfp}$ is not a $\Lambda_{\mfp}$-monomorphism. Hence $f$ is not a $\Lambda$-monomorphism. This means that $[f]=0$ by Proposition \ref{Prop:nonzeroiffmono}. Therefore $\varphi_{M}$ is injective.
		
	In order to show the surjectivity of $\varphi_{M}$, take a nonzero $\Lambda_{\mfp}$-homomorphism $g:S(P)\to M_{\mfp}$. By Lemma \ref{Lem:localinj} (\ref{item:injlocalsplit}), the canonical $\Lambda$-homomorphism $E_{\Lambda}(M)\to E_{\Lambda}(M)_{\mfp}$ is a split $\Lambda$-epimorphism. Denote a section of it by $\nu:E_{\Lambda}(M)_{\mfp}\to E_{\Lambda}(M)$ and the composite of $g:S(P)\to M_{\mfp}$, $M_{\mfp}\hookrightarrow E_{\Lambda}(M_{\mfp})$, and $\nu:E_{\Lambda}(M)_{\mfp}\hookrightarrow E_{\Lambda}(M)$ by $g'$. Since $g'$ is nonzero, there exists a nonzero element $x$ of $S(P)$ such that $0\neq g'(x)\in E_{\Lambda}(M)$. Since $M$ is an essential $\Lambda$-submodule of $E_{\Lambda}(M)$, there exists $\lambda\in\Lambda$ such that $0\neq g'(x)\lambda\in M$. Then $g'$ induces a $\Lambda$-homomorphism $f':x\lambda\Lambda\to M$. $f'$ defines an element of $\Hom_{\Lambda}(\widetilde{P},M)$. We have the commutative diagram
	\begin{equation*}
		\newdir^{ (}{!/-5pt/@^{(}}
		\xymatrix{
			S(P)\ar[r]^{g} & M_{\mfp}\ar@{^{ (}->}[r] & E_{\Lambda}(M)_{\mfp}\ar@{^{ (}->}[r]^{\nu} & E_{\Lambda}(M) \\
			x\lambda\Lambda\ar@{^{ (}->}[u]\ar[rrr]^{f'} & & & M.\ar@{^{ (}->}[u]
		}
	\end{equation*}
	By applying $(-)_{\mfp}$ to this diagram, we obtain the commutative diagram
	\begin{equation*}
		\newdir^{ (}{!/-5pt/@^{(}}
		\xymatrix{
			S(P)\ar[r]^{g} & M_{\mfp}\ar@{^{ (}->}[r] & E_{\Lambda}(M)_{\mfp}\ar@{=}[r]^{\nu_{\mfp}} & E_{\Lambda}(M)_{\mfp} \\
			(x\lambda\Lambda)_{\mfp}\ar@{=}[u]\ar[rrr]^{f'_{\mfp}} & & & M_{\mfp}.\ar@{^{ (}->}[u]
		}
	\end{equation*}
	This shows that $\varphi_{M}([f'])=g$, and hence $\varphi_{M}$ is surjective.
	
	It is straightforward to show that $\varphi_{M}$ is functorial on $M$.
	
	(2) This follows from (\ref{item:extsimple}) and Theorem \ref{Thm:Bassext}.
\end{proof}

\begin{ex}
	Let $R$ be a commutative noetherian ring and $\Lambda$ the ring of $2\times 2$ lower triangular matrices over $R$, that is,
	\begin{equation*}
		\Lambda=\begin{bmatrix} R & 0 \\ R & R \end{bmatrix}.
	\end{equation*}
	Then for any prime ideal $\mfp$ of $R$,
	\begin{equation*}
		\Lambda_{\mfp}=\begin{bmatrix} R_{\mfp} & 0 \\ R_{\mfp} & R_{\mfp} \end{bmatrix},
	\end{equation*}
	and all the isomorphism classes of simple $\Lambda_{\mfp}$-modules are given by
	\begin{equation*}
		S_{1}(\mfp)=\begin{bmatrix} k(\mfp) & 0 \end{bmatrix},\ S_{2}(\mfp)=\frac{\begin{bmatrix} k(\mfp) & k(\mfp) \end{bmatrix}}{\begin{bmatrix} k(\mfp) & 0 \end{bmatrix}},
	\end{equation*}
	where $k(\mfp)=R_{\mfp}/\mfp R_{\mfp}$.
	By considering Theorem \ref{Thm:primesimple}, denote the prime ideal of $\Lambda$ corresponding to $S_{i}(\mfp)$ by $P_{i}(\mfp)$ for each $i\in\{1,2\}$. Then $k(\widetilde{P_{i}(\mfp)})=k(\mfp)$, and
	\begin{equation*}
		E_{\Lambda}(\widetilde{P_{1}(\mfp)})=\begin{bmatrix} E_{R}(R/\mfp) & E_{R}(R/\mfp) \end{bmatrix},\ E_{\Lambda}(\widetilde{P_{2}(\mfp)})=\frac{\begin{bmatrix} E_{R}(R/\mfp) & E_{R}(R/\mfp) \end{bmatrix}}{\begin{bmatrix} E_{R}(R/\mfp) & 0 \end{bmatrix}}.
	\end{equation*}
	
	Let $V$ be an $R$-module, and define a $\Lambda$-module $M$ by
	\begin{equation*}
		M=\begin{bmatrix} V & 0 \end{bmatrix}.
	\end{equation*}
	Then
	\begin{equation*}
		M_{\mfp}=\begin{bmatrix} V_{\mfp} & 0 \end{bmatrix},\ E_{\Lambda_{\mfp}}(M_{\mfp})=\begin{bmatrix} E_{R}(V_{\mfp}) & E_{R}(V_{\mfp}) \end{bmatrix}.
	\end{equation*}
	By Theorem \ref{Thm:extBasssimple},
	\begin{align*}
		\mu_{0}(\widetilde{P_{1}(\mfp)},M)&=\dim_{k(\mfp)}\Hom_{\Lambda_{\mfp}}(S_{1}(\mfp),M_{\mfp})=\dim_{k(\mfp)}\Hom_{R_{\mfp}}(k(\mfp),V_{\mfp})=\mu_{0}(\mfp,V),\\
		\mu_{0}(\widetilde{P_{2}(\mfp)},M)&=\dim_{k(\mfp)}\Hom_{\Lambda_{\mfp}}(S_{2}(\mfp),M_{\mfp})=0,\\
		\mu_{1}(\widetilde{P_{1}(\mfp)},M)&=\dim_{k(\mfp)}\Ext^{1}_{\Lambda_{\mfp}}(S_{1}(\mfp),M_{\mfp})=\dim_{k(\mfp)}\Ext^{1}_{R_{\mfp}}(k(\mfp),V_{\mfp})=\mu_{1}(\mfp,V),\\
		\mu_{1}(\widetilde{P_{2}(\mfp)},M)&=\dim_{k(\mfp)}\Ext^{1}_{\Lambda_{\mfp}}(S_{2}(\mfp),M_{\mfp})=\dim_{k(\mfp)}\Hom_{R_{\mfp}}(k(\mfp),V_{\mfp})=\mu_{0}(\mfp,V).
	\end{align*}
\end{ex}

\section{Acknowledgments}

The author would like to express his deep gratitude to his supervisor Osamu Iyama for his elaborated guidance. The author thanks Yuji Yoshino for his valuable comments.



\end{document}